\documentclass[a4paper, 11pt]{amsart}
\usepackage[foot]{amsaddr}
\usepackage{amsmath}
\usepackage{amsthm}
\usepackage{amssymb}
\usepackage{bm}
\usepackage{dsfont}
\usepackage{enumitem}
\usepackage{graphicx}
\usepackage[hidelinks]{hyperref}
\usepackage[capitalise]{cleveref}
    \Crefname{prop}{Proposition}{Propositions}
\usepackage{mathrsfs}
\usepackage{mathtools}
\usepackage{nameref}
\usepackage{physics}
\usepackage{stmaryrd}
\usepackage{thmtools}
\usepackage{xcolor}
\usepackage{xparse}

\usepackage{tikz}
\usetikzlibrary{cd}
\usetikzlibrary{positioning, angles, quotes}

\newtheorem{thmx}{Theorem}

\newtheorem{conjx}[thmx]{Conjecture}

\theoremstyle{plain}
\newtheorem{thm}{Theorem}[section]
\newtheorem*{thm*}{Theorem}

\newtheorem{cor}[thm]{Corollary}
\newtheorem*{cor*}{Corollary}

\newtheorem{prop}[thm]{Proposition}
\newtheorem*{prop*}{Proposition}

\newtheorem{lem}[thm]{Lemma}
\newtheorem*{lem*}{Lemma}

\newtheorem{q}[thm]{Question}
\newtheorem*{q*}{Question}

\newtheorem{conj}[thm]{Conjecture}
\newtheorem*{conj*}{Conjecture}

\theoremstyle{definition}
\newtheorem{defn}[thm]{Definition}
\newtheorem*{defn*}{Definition}

\newtheorem*{ex*}{Example}

\newtheorem{rem}[thm]{Remark}
\newtheorem*{rem*}{Remark}

\theoremstyle{plain}

    \newtheoremstyle{TheoremNum}
        {\topsep}{\topsep} 
        {\itshape} 
        {-0.2cm} 
        {\bfseries} 
        {.} 
        { }  
        {\thmname{#1}\thmnote{ \bfseries #3}}
    \theoremstyle{TheoremNum}

\Crefname{defn}{Definition}{Definitions}

\newcommand{\mc}{\mathcal}
\newcommand{\cala}{\mc{A}}

\newcommand{\GL}{\mathrm{GL}}

\newcommand{\F}{\mathsf{F}}
\newcommand{\FP}{\mathsf{FP}}

\newcommand{\gd}{\mathrm{gd}}

\newcommand{\btwo}[1]{b^{(2)}_{#1}}

\newcommand{\DF}[1]{\mathcal{D}_{\mathbb{F}{#1}}}

\newcommand{\C}{\mathbb{C}}
\newcommand{\N}{\mathbb{N}}
\newcommand{\Q}{\mathbb{Q}}
\newcommand{\R}{\mathbb{R}}
\newcommand{\Z}{\mathbb{Z}}
\newcommand{\FF}{\mathbb{F}}

\newcommand{\inv}{^{-1}}

\DeclareMathOperator{\Flag}{Flag}

\DeclareMathOperator{\lk}{lk}

\DeclareMathOperator{\st}{st}

\DeclareMathOperator{\rk}{rk}

\newcommand{\Mod}{\mathsf{Mod}}

\newcounter{comments}

\title[Homological growth of Artin kernels]{Homological growth of Artin kernels in positive characteristic}
\author{Sam P.~Fisher}
\email{sam.fisher@maths.ox.ac.uk}
\author{Sam Hughes}
\email{sam.hughes@maths.ox.ac.uk}
\address[S.~P.~Fisher and S.~Hughes]{Mathematical Institute, 
Andrew Wiles Building, 
Observatory Quarter, 
University of Oxford, 
Oxford,
OX2 6GG,
United Kingdom}
\author{Ian J.~Leary}
\email{i.j.leary@soton.ac.uk}
\address[I.~J.~Leary]{School of Mathematical Sciences,
University of Southampton,
Southampton,
SO17 1BJ,
United Kingdom}
\date{\today}


\begin{document}

\begin{abstract}
    We prove an analogue of the L\"uck Approximation Theorem in positive characteristic for certain residually finite rationally soluble (RFRS) groups including right-angled Artin groups and Bestvina--Brady groups. Specifically, we prove that the mod $p$ homology growth equals the dimension of the group homology with coefficients in a certain universal division ring and this is independent of the choice of residual chain.  For general RFRS groups we obtain an inequality between the invariants.  We also consider a number of applications to fibring, amenable category, and minimal volume entropy.
\end{abstract}

\maketitle

\section{Introduction}
A celebrated theorem of L\"uck relates the rational homology growth in degree $m$ through finite covers to the $m$th $\ell^2$-Betti number of a residually finite group. More precisely:

\begin{thm}[L\"uck, \cite{LuckApprox}] \label{thm:luck}
    Let $G$ be a residually finite group of type $\mathsf F_{m+1}$ and let $(G_n)_{n\in\N}$ be a residual chain of finite index normal subgroups. Then
    \[
        \lim_{n \to \infty} \frac{b_m(G_n; \Q)}{[G:G_n]} =  b_m^{(2)}(G),
    \]
    where $b_m^{(2)}(G)$ denotes the $m$th $\ell^2$-Betti number of $G$.
\end{thm}

An immediate consequence of L\"uck's approximation theorem is that the left-hand limit always exists and is independent of the chosen residual chain.  We remind the reader that a \emph{residual chain} is a sequence $G=G_0\geq G_1\geq \dots \geq G_n\geq\dots$ such that each $G_n$ is a finite index normal subgroup of $G$ and $\bigcap_{n\in\N} G_n=\{1\}$.  A related invariant is the \emph{$m$th $\FF_p$-homology gradient} for a finite field $\FF_p$. It is defined by 
\[
\btwo{m}(G, (G_n);\FF_p)\coloneqq \limsup_{n\to\infty}\frac{b_m(G_n;\FF_p)}{[G:G_n]} 
\]
for any $G$ of type $\mathsf{F}_{m+1}$ and residual chain $(G_n)_{n\in\N}$.  We will recall the definition of various finiteness properties in \Cref{sec:fibring}.

In \cite[Conjecture~3.4]{Luck2016survey} L\"uck conjectured that $\btwo{m}(G;\FF_p)$ should equal $\btwo{m}(G)$ (and hence be independent of the residual chain).  This was disproved by Avramidi--Okun--Schreve \cite{AvramidiOkunSchreve2021} (using a result of Davis--Leary \cite{DavisLeary2003}) where they showed that 
\[\btwo{3}(A_{\mathbb{R}\mathbf{P}^2})=0 \quad \text{but} \quad \btwo{3}(A_{\mathbb{R}\mathbf{P}^2};\FF_2)=1 \]
independently of the choice of residual chain.  Here, $A_{\mathbb{R}\mathbf{P}^2}$ is the right-angled Artin group (RAAG) on (the $1$-skeleton of) any flag triangulation of the real projective plane.

For torsion-free groups satisfying the Atiyah conjecture, the $\ell^2$-Betti numbers may be computed via the dimensions of group homology with coefficients in a certain skew field $\mc D_{\Q G}$, known as the \emph{Linnell skew field} of $G$.  In \cite{JaikinZapirain2020THEUO}, Jaikin-Zapirain introduces analogues of the Linnell skew field with ground ring any skew field $\FF$, denoted $\mc D_{\FF G}$, and called the \textit{Hughes-free division ring} of $\mathbb F G$. He proves that $\mathcal D_{\mathbb F G}$ exists and is unique up to $\FF G$-isomorphism for large classes of groups, including residually finite rationally soluble (RFRS) groups (in particular compact special groups) and conjectures they should exist for all locally indicable groups.

One may compute group homology with coefficients in $\mc D_{\FF G}$ and take $\mc D_{\FF G}$-dimensions to obtain $\mc D_{\FF G}$-Betti numbers, denoted $b_m^{\mc D_{\FF G}}(G)$. We emphasise that when $G$ is RFRS and $\mathbb F = \Q$, the Linnell skew field and the Hughes free division ring coincide \cite[Appendix]{JaikinZapirain2020THEUO}, and therefore that $b_m^{\mathcal D_{\Q G}}(G) = b_m^{(2)}(G)$. The $\DF{G}$-Betti numbers share a number of properties with $\ell^2$-Betti numbers (see for example \cite[Theorem~3.9]{HennekeKielak2021} and \cite[Lemmas 6.3 and 6.4]{Fisher2021improved}). In light of this we raise the following conjecture:

\begin{conjx}\label{main conjecture}
    Let $\FF$ be a skew field.  Let $G$ be a torsion-free residually finite group of type $\mathsf{FP}_{n+1}(\mathbb F)$ such that $\mc D_{\FF G}$ exists.  Let $(G_i)_{i\in \N}$ be a residual chain of finite index normal subgroups.  Then,
    \[
        \btwo{m}(G,(G_i);\FF) = b^{\DF{G}}_m(G)
    \]
    for all $m \leqslant n$. In particular, the limit supremum in the definition of $\btwo{m}$ is a genuine limit and is independent of the choice of residual chain.
\end{conjx}

Our main result verifies this conjecture for various families of groups---notably for RAAGs, which provided the counterexamples to L\"uck's original conjecture.  

\begin{thmx}\label{main theorem}
Let $\FF$ be a skew field and $G$ be a group commensurable with any of
\begin{enumerate}[label=(\arabic*)]
    \item\label{item:Artin} a residually finite Artin group satisfying the $K(\pi,1)$ Conjecture, such as a right-angled Artin group or RAAG.
    \item\label{item:ArtinKer} an Artin kernel, i.e.~the kernel of a homomorphism from a RAAG to $\Z$ (these include Bestvina--Brady groups);
    \item\label{item:graphProd} a graph product of amenable RFRS groups.
\end{enumerate}  If $G$ is type $\mathsf{FP}_n(\mathbb F)$, then $\btwo{m}(G, (G_n);\FF)=b^{\DF{G}}_m(G)$ for $m \leqslant n$.  In particular, the limit supremum in the definition of $\btwo{m}$ is a genuine limit independent of the choice of residual chain $(G_n)_{n\in\N}$.
\end{thmx}

In the case when $G$ is not torsion-free but contains a finite-index torsion-free subgroup $H$, each of Betti number $b_i(G)$ for $G$ appearing in the statement is defined to be $b_i(H)/[G:H]$.  This extension to the definition is clearly consistent for $\btwo{i}(G;\FF)$ and is consistent for $b^{\DF{G}}_i(G)$ because  of~\cite[Lemma~6.3]{Fisher2021improved}.  

In each case we are able to compute $b^{\mc D_{\FF G}}_m(G)$ explicitly; in cases \ref{item:Artin} and \ref{item:graphProd} we find that it is equal to the $\mathbb F$-homology gradient previously computed in \cite{AvramidiOkunSchreve2021} and \cite{OkunSchreveTorsion}. In case \ref{item:ArtinKer} we compute both the $\DF{G}$-Betti numbers and the $\mathbb F$-homology gradients and show they are equal. We highlight the computation for \ref{item:ArtinKer}, which we expect will be of independent interest.  Note that this generalises the computation of Davis and Okun for the $\ell^2$-Betti numbers of Bestvina--Brady groups \cite{DavisOkun2012}.

\begin{thmx}\label{thmx BBL}
Let $\FF$ be a skew field, let $\varphi \colon A_L \rightarrow \Z$ be an epimorphism and let $BB_L^\varphi$ denote $\ker\varphi$. If $BB_L^\varphi$ is of type $\mathsf{FP}_n(\mathbb F)$ then
    \[
        b_m^{\DF{BB_L^\varphi}}(BB_L^\varphi) = b_m^{(2)}(BB_L^\varphi; \mathbb F) = \sum_{v \in L^{(0)}} \abs{\varphi(v)} \cdot \widetilde{b}_{m-1}(\lk(v); \mathbb F).
    \]
    for all $m \leqslant n$.
\end{thmx}

One may extend \Cref{main conjecture} to $G$-spaces with finite $(n+1)$-skeleton.  In this more general setting we are able to verify the conjecture for certain polyhedral product spaces (\Cref{thm:agrGraphProd}) and certain hyperplane arrangements (\Cref{hyperplanes}).  We remark that \cite{LinnellLuckSauer} effectively proves the conjecture for torsion-free amenable groups.

For groups where groups where $\DF{G}$ exists and is \textit{universal} (see \cref{sec:kaz} for a definition), we obtain that the agrarian Betti-numbers give a lower bound for the homology gradients as an easy consequence of a result of Jaikin-Zapirain \cite[Corollary 1.6]{JaikinZapirain2020THEUO}. Note that the following theorem applies to all residually (amenable and locally indicable) groups by \cite[Corollary 1.3]{JaikinZapirain2020THEUO}, and in particular to RFRS groups.

\begin{thmx}\label{thm:lowerBound}
    Let $\mathbb F$ be a skew-field and let $G$ be a residually finite group of type $\FP_{n+1}(\mathbb F)$ such that $\DF{G}$ exists and is the universal division ring of fractions of $\mathbb FG$. Then 
    \[
        b_m^{\DF{G}}(G) \leqslant b_m^{(2)}(G, (G_i); \mathbb F)
    \]
    for all $m \leqslant n$, where $(G_i)_{i \in \N}$ is any residual chain of finite-index subgroups of $G$.
\end{thmx}

We also mention the work of Bergeron--Linnell--L\"uck--Sauer \cite{BLLS_p_analytic}, which we believe provides some more evidence for \cref{main conjecture}. Let $\Gamma$ be the fundamental group of a finite CW complex $X$ with a homomorphism $\varphi \colon \Gamma \rightarrow \GL_n(\Z_p)$. Let $G$ be the closure of $\varphi(\Gamma)$ and let $G_i = \ker(G \rightarrow \GL_n(\Z/p^i\Z))$. Recall that the \textit{Iwasawa algebra} of $G$ over $\mathbb F_p$ is $\mathbb F_p\llbracket G \rrbracket = \varprojlim \mathbb F_p[G/G_i]$. If $G$ is torsion-free, then $\mathbb F_p\llbracket G \rrbracket$ has no zero divisors and is Ore with respect to its nonzero elements. Letting $\overline{\Gamma} = \Gamma / \ker \varphi$, we have a ring homomorphism
\[
    \mathbb F \overline\Gamma \rightarrow \mathbb F G \rightarrow \mathbb F_p\llbracket G \rrbracket \rightarrow \mathcal D,
\]
where $\mathcal D$ is the Ore localisation of $\mathbb F_p\llbracket G \rrbracket$. If $M$ is an $\mathbb F_p\llbracket G \rrbracket$-module and $G$ is torsion-free, then the dimension of $M$ is
\[
    \dim_{\mathbb F_p\llbracket G \rrbracket} M := \dim_\mathcal D (\mathcal D \otimes_{\mathbb F_p\llbracket G \rrbracket} M).
\]
If $G$ is not torsion-free, then we can pass to a uniform finite-index subgroup $G_0 \leqslant G$ and then define $\dim_{\mathbb F_p\llbracket G \rrbracket} M = [G:G_0]\inv \dim_{\mathbb F_p\llbracket G \rrbracket} M$. There is then a natural mod $p$ analogue of $\ell^2$-Betti numbers given by
\[
    \beta_k(\overline{X}, \overline{\Gamma}; \FF_p) = \dim_{\mathbb F_p\llbracket G \rrbracket} H_k( \mathbb F_p\llbracket G \rrbracket \otimes_{\FF_p\overline{\Gamma}} C_\bullet(\overline{X}; \FF_p))
\]
where $\overline X$ is the cover of $X$ corresponding to $\ker \varphi$. With this setup, Bergeron--Linnell--L\"uck--Sauer prove the following mod-$p$ L\"uck approximation style theorem.

\begin{thm}
    With the notation above, let $\Gamma_i = \varphi\inv(G_i)$ and let $X_i$ be the corresponding cover of $X$. Then
    \[
        b_k(X_i; \FF_p) = [\Gamma : \Gamma_i]  \cdot \beta_k(\overline X, \overline{\Gamma}; \FF_p) + O\left( [\Gamma:\Gamma_i]^{1 - \frac{1}{\dim G}} \right)
    \]
    for all $k$.
\end{thm}

\subsection{Outline of the paper}  
In \Cref{sec:prelims} we give the relevant background on group rings and the computational tools we will need.  In fact in many computations we are able to work in the more general setting of \emph{agrarian invariants} as defined in \cite{HennekeKielak2021} and so summarise the relevant theory.  The remaining three sections of the paper are described in the sequel.

\subsubsection{Computations}
In \Cref{sec:algcomps} we introduce the notion of a \emph{confident complex}; roughly this is a CW complex admitting a finite cover by open sets with amenable fundamental group such that the nerve of the cover has good properties.  We then compute the $\DF{G}$-Betti numbers and $\FF_p$-homology gradients of these complexes showing they are related to $\FF_p$-Betti numbers of the nerve.  The computation of the first invariant a uses a spectral sequence collapsing result of Davis--Okun \cite{DavisOkun2012} building on work of Davis--Leary \cite{DavisLeary2003}.  The computation of the second invariant is similar in spirit to the work of Avramidi--Okun--Schreve \cite{AvramidiOkunSchreve2021} and Okun--Schreve \cite{OkunSchreveTorsion} and again relies on a spectral sequence argument.

The remainder of the section involves computations of the homological invariants for various spaces and groups with the goal of showing they satisfy \Cref{main conjecture}. The computations are summarised as follows. In \Cref{thm:BBapprox} we show that Artin kernels of type $\FP_n$ are fundamental groups of confident spaces with an $n$-acyclic covering space, and use this to prove \Cref{thmx BBL}.  In \Cref{thm:agrGraphProd} we compute the invariants for graph products of amenable RFRS groups (including RAAGs) and polyhedral products of classifying spaces of amenable RFRS groups.  In \Cref{thm Artin groups} we compute the invariants for the Artin groups alluded to earlier. In fact the method applies to RFRS groups admitting a strict fundamental domain with certain stabilisers (\Cref{sec strict fund}). Finally, inspired by \cite{DavisJanuszkiewiczLeary2007}, we compute the invariants for hyperplane arrangements whenever the invariants are defined (\Cref{hyperplanes}).

\subsubsection{A lower bound for homology gradients}

In \Cref{sec:kaz}, we recall the notion of a universal division ring of fractions and prove \cref{thm:lowerBound} as a consequence of work of Jaikin-Zapirain in \cite{JaikinZapirain2020THEUO}.

\subsubsection{Applications to fibring}

In \Cref{thm no fibring} we prove that if a RFRS group of type $\FP_{n+1}(\FF)$ is not virtually $\FP_n(\FF)$-fibred, then there is some $m\leqslant n$ such that $b_m^{(2)}(G,(G_n);\FF)>0$ for every residual chain $(G_n)_{n\in\N}$.  

In the remainder of \cref{sec:fibring} we apply the computations of the agrarian invariants of RAAGs and Artin kernels to obtain some results about fibring in RAAGs. In particular, we make progress towards the following question of Matthew Zaremsky, communicated to us by Robert Kropholler: If a RAAG virtually algebraically fibres with kernel of type $\F_n$, then does it algebraically fibre with kernel of type $\F_n$? We are able to answer this question if one replaces $\F_n$ with $\FP_n(R)$, where $R$ is either a skew field, $\Z$, or $\Z/m$ for some integer $m > 1$ (\cref{thm:virtFibiffFib}). Since finitely presented groups of type $\FP_n$ are of type $\F_n$, this leaves $\F_2$ as the main case of interest in Zaremsky's question.

In \cite{Fisher2021improved}, the first author showed that if $G$ is RFRS and $\ell^2$-acyclic in dimensions $\leqslant n$, then $G$ virtually $\FP_n(\Q)$-fibres, but left unanswered whether $\ell^2$-Betti numbers control virtual $\FP_n$-fibring. We resolve that here by showing that there are RAAGs that are $\ell^2$-acyclic but do not virtually $\FP_2$-fibre (\cref{prop:noZfibre}). Finally, we use the explicit computation of the agrarian invariants of Artin kernels to show that if $A_L$ is a RAAG and $\FF$ is a skew field, then either all of the $\FP_n(\FF)$-fibres of $A_L$ are themselves virtually $\FP_n(\FF)$-fibred or none of them are (\cref{thm:fibresFibre}, \cref{cor:fibresFibre}).

\subsubsection{Applications to amenable category and minimal volume entropy}
In \Cref{sec min vol entropy} we relate the Agrarian invariants to the amenable category of \cite{CapovillaLohMoraschini2022} and the minimal volume entropy of Gromov \cite{Gromov1982volBC}. 
 The relevant background is described in the section.  
 
 We show via an argument of Sauer \cite{Sauer2016} that having a small amenable category implies vanishing of $\DF{G}$-Betti numbers for residually finite groups (\Cref{AMN cat vs DFG}).  We also show that for residually finite groups with uniformly uniform exponential growth admitting a finite $K(G,1)$, having a non-zero $\DF{G}$-Betti number implies the minimal volume entropy of $G$ is non-zero (\Cref{cor min vol}).  The analogous results for $\FF_p$-homology gradients were established by Sauer \cite{Sauer2016} and Haulmark--Schreve \cite{HaulmarkSchreve2022minimal} respectively.  In some sense this provides more evidence towards \Cref{main conjecture}.  Finally, we give a condition for the minimal volume entropy of an Artin kernel admitting a finite $K(G,1)$ to be non-zero (\Cref{min vol BBL}) and conjecture a converse.

\subsection*{Acknowledgements}
The first author is supported by the National Science and Engineering Research Council (NSERC) [ref.~no.~567804-2022]. The second received funding from the European Research Council (ERC) under the European Union's Horizon 2020 research and innovation programme (Grant agreement No. 850930). The authors would like to thank Ismael Morales for asking a question that led to \cref{thm:fibresFibre}, as well as Robert Kropholler and Matt Zaremsky for communicating a question which inspired \Cref{thm:virtFibiffFib}.  The authors would like to thank Dawid Kielak for helpful conversations.

\section{Preliminaries}\label{sec:prelims}

Throughout all rings are assumed to be associative and unital.

\subsection{Finiteness properties} \label{subsec:finprops}

Let $R$ be a ring and $G$ be a group. Then $G$ is said to be of \textit{type $\FP_n(R)$} if there is a projective resolution $P_\bullet \rightarrow R$ of the trivial $RG$-module $R$ such that $P_i$ is finitely generated for all $i \leqslant n$. If $G$ is of type $\FP_n(R)$ and $S$ is an $R$-algebra, then $G$ is of type $\FP_n(S)$. Thus, if $G$ is of type $\FP_n(\Z)$, then $G$ is of type $\FP_n(R)$ for any ring $R$; because of this, we write $\FP_n$ to mean $\FP_n(\Z)$. Note that finite generation is equivalent to $\FP_1(R)$ for any ring $R$, though $\FP_2$ is in general a stronger condition than $\FP_2(R)$.

We also the mention the homotopical analogue of the $\FP_n(R)$ condition: a group $G$ is of \textit{type $\F_n$} if $G$ has a classifying space with finite $n$-skeleton. Note that $\F_1$ is equivalent to $\FP_1(R)$ for any ring $R$, but that $\F_n$ is in general strictly stronger than $\FP_n(R)$ for $n \geqslant 2$.

\subsection{Agrarian invariants}

Let $G$ be a group and let $\mathbb F$ be a skew field. The \textit{group ring} $\mathbb F G$ is the set of formal sums $\sum_{g \in G} \lambda_g g$, where $\lambda_g \in \mathbb F$ is zero for all but finitely many $g \in G$, equipped with the obvious addition and multiplication operations. Let $\mathcal D$ be a skew field. Then an \textit{agrarian embedding} is a ring monomorphism $\alpha \colon \mathbb F G \hookrightarrow \mathcal D$. Agrarian embeddings were first studied by \cite{Malcev1948, Neumann1949}, who proved that group rings of bi-orderable groups have agrarian embeddings. Note that the existence of an agrarian embedding implies that $G$ is torsion-free and that the Kaplansky zero divisor and idempotent conjectures hold for $\mathbb F G$. There is no known example of a torsion-free group $G$ and a skew field $\mathbb F$ such that $\mathbb F G$ does not have an agrarian embedding.

Let $X$ be a CW-complex with a cellular $G$ action such that for every $g \in G$ and every open cell $e$ of $X$, if $g \cdot e \cap e \neq \varnothing$ then $g$ fixes $e$ pointwise. The cellular chain complex $C_\bullet(X; \mathbb F)$ is naturally an $\mathbb FG$-module. In this situation, $X$ is called a $G$-CW complex. If $\alpha \colon \mathbb F G \rightarrow \mathcal D$ is an agrarian map, then we can define the \textit{$\mathcal D$-homology} and \textit{$\mathcal D$-Betti numbers} of $X$ by 
\[
    H_p^\mathcal D (X) := H_p(\mathcal D \otimes_{\mathbb FG} C_\bullet(X)) \qquad \text{and} \qquad b_p^\mathcal{D}(X) = \dim_\mathcal D H_p^\mathcal D(X)
\]
where $\dim_\mathcal D$ denotes the dimension as a $\mathcal D$-module, which is well-defined since $\mathcal D$ is a skew field. Taking $X$ to be a classifying space of $G$, we obtain the $\mathcal D$-homology and $\mathcal D$-Betti numbers of $G$.

The following theorem gives a central example of an agrarian embedding.

\begin{thm}[Linnell \cite{LinnellDivRings93}]
    If $G$ is a torsion-free group satisfying the \textnormal{strong Atiyah conjecture}, then there is a skew field $\mc D_{\Q G}$, known as the \textnormal{Linnell skew field} of $G$, and an agrarian embedding $\mathbb Q G \hookrightarrow \mathcal D_{\Q G}$ such that $b_p^{(2)}(G) = b^{\mathcal D_{\Q G}}(G)$ .
\end{thm}

The strong Atiyah conjecture asserts that if $X$ is a free $G$-CW complex of finite type and $G$ has finite subgroups of bounded order, then $\operatorname{lcm}(G) \cdot b_p^{(2)}(X) \in \Z$ for all $p \in \N$, where $\operatorname{lcm}(G)$ is the least common multiple of the orders of finite subgroups of $G$. The strong Atiyah conjecture is known for many groups, in particular for residually (torsion-free solvable groups) \cite{SchickL2Int2002}, for cocompact special groups \cite{SchreveSpecialAtiyah}, and for locally indicable groups \cite{ZapirainLopezStrongAtiyah2020}. Importantly for us, the Atiyah conjecture holds for all RFRS groups, and in particular all subgroups of RAAGs. The following theorem of Jaikin-Zapirain provides many examples of agrarian embeddings in positive characteristic.

\begin{thm}[Jaikin-Zapirain {\cite[Theorem 1.1]{JaikinZapirain2020THEUO}}]\label{thm:jaikin}
    Let $\mathbb F$ be a skew field and let $G$ be either locally indicable amenable, residually (torsion-free nilpotent), or free-by-cyclic. Then there exists a division ring $\mathcal D_{\mathbb F G}$, known as the \textnormal{Hughes-free division ring} of $\mathbb FG$, and an agrarian embedding $\mathbb F G \hookrightarrow \mathcal D_{\mathbb F G}$.
\end{thm}

\subsection{A Mayer--Vietoris type spectral sequence}\label{subsec:spsq}

The following construction of a Mayer--Vietoris type spectral sequence is due to Davis and Okun \cite{DavisOkun2012}. We will state the homological version of the spectral sequence with arbitrary coefficients. Let $\mathsf{P}$ be a poset. We define $\Flag(\mathsf{P})$ to be the simplicial realisation of $\mathsf{P}$, i.e.~$\Flag(\mathsf{P})$ is the simplicial complex whose simplices are the totally ordered, finite, nonempty subsets of $\mathsf{P}$. Hence, every simplex $\sigma \in \Flag(\mathsf{P})$ has a well-defined \textit{minimum vertex}, denoted $\min(\sigma)$.

If $Y$ is a CW complex, then a \textit{poset of spaces} in $Y$ over $\mathsf{P}$ is a cover $\mathcal{Y} = \{Y_a\}_{a \in \mathsf{P}}$ of $Y$ with each $Y_a$ a subcomplex such that
\begin{enumerate}
    \item $a < b$ implies $Y_a \subseteq Y_b$;
    \item $\mathcal{Y}$ is closed under finite, nonempty intersections.
\end{enumerate}

Let $R$ be a ring and let $\Mod_R$ denote the category $R$-modules. A \textit{poset of coefficients} for $\mathsf{P}$ is a contravariant functor $\mathcal{A} \colon \mathsf{P} \rightarrow \Mod_R$. The functor $\mathcal{A}$ induces a system of coefficients on  $\Flag(\mathsf{P})$ via $\sigma \mapsto \mathcal{A}_{\min(\sigma)}$, which gives chain complex
\[
    C_j(\Flag(\mathsf{P}); \mathcal{A}) := \bigoplus_{\sigma \in \Flag(\mathsf{P})^{(j)}} \mathcal{A}_{\min(\sigma)},
\]
where $\Flag(\mathsf{P})^{(j)}$ is the set of $j$-simplices in $\Flag(\mathsf{P})$.

\begin{lem}[{\cite[Lemmas 2.1 and 2.2]{DavisOkun2012}}] \label{lem:spec}
Let $M$ be an $R$-module and suppose $\mathcal{Y} = \{Y_a\}_{a \in \mathsf{P}}$ is a poset of spaces over $\mathsf{P}$ in a CW complex $Y$. There is a Mayer--Vietoris type spectral sequence
\[
    E_{p,q}^2 = H_p(\Flag(\mathsf{P}); \mathcal{H}_q(\mathcal{Y}; M)) \rightarrow H_{p+q}(Y;M),
\]
where $\mathcal{H}_\bullet(\mathcal{Y}; M)$ is a system of coefficients given by 
\[
    \mathcal{H}_\bullet(\mathcal{Y}; M)(\sigma) = H_\bullet(Y_{\min(\sigma)}; M).
\]
Moreover, if the induced homomorphism $H_\bullet(Y_a; M) \rightarrow H_\bullet(Y_b; M)$ is zero whenever $a < b$ in $\mathsf{P}$, then
\[
    E_{p,q}^2 = \bigoplus_{a \in \mathsf{P}} H_p(\Flag(\mathsf{P}_{\geqslant a}), \Flag(\mathsf{P}_{> a}) ; H_q(Y_a ; M)).
\]
\end{lem}

\section{Computations} \label{sec:algcomps}

\subsection{Approximation for spaces with confident covers}\label{subsec:confident}

Fix a skew field $\mathbb F$ and let $X$ be a compact CW complex with a finite poset of spaces $\mathcal X = \{X_\alpha\}_{\alpha \in \mathsf P}$ over $\mathsf P$. We call $\mathcal X$ \textit{confident} if it satisfies the following conditions: for each $\alpha \in \mathsf P$
\begin{enumerate}[label = (\roman*)]
    \item each $X_\alpha$ has finitely many components; 
    \item\label{item:amen} either $X_\alpha \subseteq X^{(0)}$ or each connected component of $X_\alpha$ is a classifying space with torsion-free amenable fundamental group such that $\mathbb F G$ has no zero-divisors;
    \item if $C \subseteq X_\alpha$ is a component, then the inclusion $C \subseteq X$ induces an injection $\pi_1(C) \rightarrow \pi_1(X)$;
    \item\label{item:separation} if $X_\alpha$ is a collection of points, then $\alpha$ is minimal in $\mathsf P$; equivalently, $X_\alpha \cap X_\beta = \varnothing$ whenever $X_\alpha, X_\beta \subseteq X^{(0)}$.
\end{enumerate}

\begin{rem}
    If $G$ is a torsion-free elementary amenable group, then $\mathbb F G$ has no zero-divisors.
\end{rem}

We will show that spaces with confident covers satisfy L\"uck's approximation theorem in arbitrary characteristic, following Avramidi, Okun, and Schreve who prove the same result for the Salvetti complex of a RAAG. We will then show that the result agrees with the agrarian Betti numbers of the space.

Before beginning, we fix some notation. For the rest of the section, $X$ will be a compact CW complex with a confident cover $\mathcal X = \{X_\alpha\}_{\alpha\in\mathsf P}$ and we assume that $G := \pi_1(X)$ is residually finite with residual chain $(G_n)_{n \in \N}$. Let $X_n$ be the covering space of $X$ corresponding to $G_n \trianglelefteqslant G$. Fix some $n \in \N$ and let $V = \mathbb F[G/G_n]$. Let $K$ be the nerve of $\mathcal X$ and let $X_\sigma = X_{\min(\sigma)}$ for any simplex $\sigma \in K$. By an abuse of notation, we will use $\alpha$ to denote both an element of $\mathsf P$ and the corresponding vertex of $K$. 

Recall the Mayer--Vietoris type spectral sequence 
\[
    E_{p,q}^1 = C_p(K; H_q(X_\sigma; V)) \Rightarrow H_{p+q}(X; V) \cong H_{p+q}(X_n; \mathbb F)
\]
(see, e.g., \cite[VII.4]{BrownGroupCohomology}).

\begin{lem}\label{lem:growth}
    We have that
    \[
        \lim_{n \to \infty} \frac{\dim_\mathbb F H_q (X_\sigma;  V)}{[G : G_n]} = \begin{cases}
            n_\sigma & \text{if} \ q = 0 \ \text{and} \ X_\sigma \subseteq X^{(0)}, \\
            0 & \text{otherwise}.
        \end{cases}
    \]
\end{lem}

\begin{proof}
    The proof is similar to that of \cite[Lemma 8]{AvramidiOkunSchreve2021}. The claim is clear when $X_\sigma$ consists of $0$-cells. In the other case, since the homology growth of amenable groups satisfying \ref{item:amen} is sublinear \cite[Theorem 0.2]{LinnellLuckSauer}, the only way $\dim_\mathbb F H_q (X_\sigma; V)$ can grow linearly is if the number of components of the preimage of $X_\sigma$ in $X_n$ grows linearly with the index. But this does not occur since the sequence $\Gamma_n$ is residual and normal and the inclusions $X_\sigma \subseteq X$ induce $\pi_1$-injections of infinite groups on each component of $X_\alpha$. \qedhere
\end{proof}

The spectral sequence is therefore concentrated on the $E_{p,0}^1$ line, up to an error sublinear in the index $[G:G_n]$. This implies that
\begin{equation}\label{eq:limsup1}
    \limsup_{n \in \N} \frac{\dim_\mathbb F E_{p,0}^2}{[G:G_n]} = \limsup_{n \in \N} \frac{\dim_\mathbb F H_p(X;V)}{[G:G_n]}.
\end{equation}
Define a poset of coefficients on the vertices of $K$ by 
\[
    \mathcal A_\alpha = \begin{cases}
        V^{n_\sigma} & \text{if} \ X_\sigma \subseteq X^{(0)}, \\
        0 & \text{otherwise}.
    \end{cases}
\]
Let $K^{(p)}$ denote the set of $p$-simplices of $K$. There is a chain projection 
\[
    E_{p,0}^1 = \bigoplus_{\sigma \in K^{(p)}} V^{n_\sigma} \rightarrow C_p(K; \mathcal A_\alpha) = \bigoplus_{\sigma \in K^{(p)} : X_{\min(\sigma)} \subseteq X^{(0)}} V^{n_\sigma}.
\]
By \cref{lem:growth}, the kernel of this projection has dimension sublinear in the index $[G:G_n]$ and therefore
\begin{equation}\label{eq:limsup2}
    \limsup_{n \in \N} \frac{\dim_\mathbb F E_{p,0}^2}{[G:G_n]} = \limsup_{n \in \N} \frac{\dim_\mathbb F H_p(X;\mathcal A_\alpha)}{[G:G_n]}.
\end{equation}

The proof of the following proposition is similar to that of \cite[Lemma 9]{AvramidiOkunSchreve2021}, except that in their case the nerve is contractible. Though $K$ is not necessarily contractible, it does decompose nicely into contractible pieces centred at the vertices $\alpha \in K^{(0)}$ such that $X_\alpha \subseteq X^{(0)}$.

\begin{prop}\label{prop:nerveHomology}
    Let $S = \{\alpha \in K^{(0)} : X_\alpha \subseteq X^{(0)}\}$. Then 
    \[
        \dim_\mathbb F H_p(K; \mathcal A_\sigma) = [G:G_n] \cdot \sum_{\alpha \in S} n_\alpha \widetilde{b}_{p-1}(\lk(\alpha); \mathbb F)
    \]
    In particular, $\lim_{n \to \infty} \frac{b_p(X_n; \mathbb F)}{[G:G_n]}$ exists and is independent of the residual chain $(G_n)$.
\end{prop}

\begin{proof}
    By \ref{item:separation}, if $\alpha \in K^{(0)}$ is a vertex such that $\mathcal A_\alpha \neq 0$, then every vertex $\beta \in K^{(0)}$ adjacent to $\alpha$ has $\mathcal A_\beta = 0$. Therefore the chain complex $C_\bullet (K; \mathcal A_\sigma)$ decomposes as a direct sum of chain complexes $\bigoplus_{\alpha \in S} C_\bullet (\st(\alpha); \mathcal A_\sigma)$, where the coefficient system $\mathcal A_\sigma$ is restricted to each $\st(\alpha) \subseteq K$.

    For each $\alpha \in S$, there is a short exact sequence of chain complexes
    \[
        0 \rightarrow C_\bullet(\lk(\alpha); \mathbb F) \otimes V^{n_\alpha} \rightarrow C_\bullet (\st(\alpha); \mathbb F) \otimes V^{n_\alpha} \rightarrow C_\bullet (\st(\alpha); \mathcal A_\sigma) \rightarrow 0.
    \]
    Because $\st(\alpha)$ is contractible, the middle term is acyclic and therefore
    \[
        H_\bullet(\st(\alpha); \mathcal A_\sigma) \cong H_{\bullet-1}(\lk(\alpha); \mathbb F) \otimes V^{n_\alpha}.
    \]
    The formula in the statement of the proposition follows, since $V^{n_\alpha}$ is a vector space of dimension $[G:G_n] \cdot n_\alpha$. 
    
    This formula together with \eqref{eq:limsup1} and \eqref{eq:limsup2} show that $\limsup_n \frac{b_p(X_n; \mathbb F)}{[G:G_n]}$ is independent of the residual chain $(G_n)$. Moreover, this implies that the $\limsup$ is a genuine limit. \qedhere
\end{proof}

\begin{rem}
    In \cite{AvramidiOkunSchreve2021}, the computation of \cref{prop:nerveHomology} is carried out in the case that $X_L$ is the Salvetti complex of the RAAG determined by $L$. If $\mathcal X$ is the cover of $X_L$ by standard tori, then there is a single vertex $v$ in the cover and the corresponding vertex $\alpha$ in the nerve has link isomorphic to $L$. Thus, we recover the formula $b_p^{(2)}(A_L; \mathbb F) = \widetilde{b}_{p-1}(L;\mathbb F)$.
\end{rem}

\begin{rem}\label{rem gradient spaces}
    Condition~\ref{item:amen} can be weakened as follows:  One only requires that first, $\pi_1X$ is residually finite; and second, that each $X_\sigma$ is (homotopy equivalent to) a compact CW complex with vanishing $\FF$-$\ell^2$-Betti numbers independent of the chain, or is a $0$-cell.  In this case the conclusion of \Cref{prop:nerveHomology} still holds.
\end{rem}

\begin{cor}\label{cor:FIconfident}
    Let $X$ be a confident CW complex with $\pi_1(X)$ residually finite. If $\pi \colon \widehat{X} \rightarrow X$ is a degree $d$ cover, then $b_p^{(2)}(H; \mathbb F) = d \cdot b_p^{(2)}(G; \mathbb F)$.
\end{cor}

\begin{proof}
    Let $\{X_\alpha\}$ be a confident cover of $X$. Then $\{ \pi\inv(X_\alpha) \}$ is a confident cover of $\widehat{X}$ and $\abs{\pi\inv(X_\alpha)} = d \cdot \abs{X_\alpha}$ whenever $X_\alpha \subseteq X^{(0)}$. \qedhere
\end{proof}

\subsection{Agrarian homology of spaces with confident covers}\label{subsec:agrarianConfidence}

We continue with the same set-up as the previous subsection: $X$ is a CW complex with a confident cover $\mathcal X = \{X_\alpha\}_{\alpha \in \mathsf P}$ and let $K$ be the covering. Notice that $K \cong \operatorname{Flag}(\mathsf P)$. Additionally, we will assume that there exists a skew field $\mathcal D$ and a fixed agrarian embedding $\mathbb FG \rightarrow \mathcal D$, where $G = \pi_1(X)$.

\begin{prop}\label{prop:nerveAgrarian}
    Let $S = \{ \alpha \in K^{(0)} : X_\alpha \subseteq X^{(0)} \}$. Then
    \[
        \dim_\mathbb F H_p^\mathcal D(X) = \sum_{\alpha \in S} n_\alpha \widetilde{b}_{p-1} (\lk(\alpha); \mathbb F).
    \]
    In particular, if $\pi_1(X)$ is residually finite then $b_p^\mathcal{D}(X) = b_p^{(2)} (X; \mathbb F)$.
\end{prop}

\begin{proof}
    Suppose $X_\alpha$ does not consist of $0$-cells. Then, each component of $X_\alpha$ is a classifying space for an infinite amenable group and therefore $H_p^\mathcal D(X_\alpha) = 0$ by \cite[Theorem 3.9(6)]{HennekeKielak2021}. Since the elements $\alpha$ such that $X_\alpha \subseteq X^{(0)}$ are minimal in $\mathsf P$, the spectral sequence of \cref{lem:spec} collapses on the $E_{p,0}^2$ line. By \cref{lem:spec},
    \[
        H_p^\mathcal D(X) = \bigoplus_{\alpha \in S} H_p(\Flag(\mathsf P_{\geqslant \alpha}), \Flag(\mathsf P_{>\alpha}); \mathcal D^{n_\alpha}),
    \]
    whence the stated formula follows.
\end{proof}

\begin{cor}\label{cor:virtAgrarian}
    If $X$ is confident, $G = \pi_1(X)$ is residually finite, and $G$ has a Hughes-free division ring $\DF{G}$, then $b_p^{\DF{H}} = b_p^{(2)}(\widehat X; \mathbb F)$ for every finite index subgroup $H \leqslant G$ and corresponding finite covering space $\widehat X \rightarrow X$.
\end{cor}

\begin{proof}
    This follows from \cref{cor:FIconfident} and the fact that Hughes-free Betti numbers scale when passing to a finite index subgroup \cite[Lemma 6.3]{Fisher2021improved}. \qedhere
\end{proof}

\begin{rem}\label{rem agrarian spaces}
    Condition~\ref{item:amen} can be weakened as follows:  One only requires that first $\pi_1X$ admits an agrarian map $\pi_1 X\to \mc D$; and second that each $X_\sigma$ is (homotopy equivalent to) a compact CW complex with with vanishing $\mc D$-Betti numbers or is a $0$-cell.  In this case the conclusion of \Cref{prop:nerveAgrarian} still holds.
\end{rem}

\subsection{Artin kernels}

An Artin kernel is simply the kernel of a non-zero homomorphism $\varphi \colon A_L \rightarrow \Z$, where $L$ is a flag complex and $A_L$ is the RAAG it determines. We now apply the results of the previous two subsections to Artin kernels and obtain an explicit formula for their agrarian Betti numbers in terms of $L$ and $\varphi$. 

We fix a skew field $\mathbb{F}$, a flag complex $L$, and a surjective homomorphism $\varphi \colon A_L \rightarrow \Z$, where $A_L$ is the RAAG on $L$. Moreover, we fix a standard generating set for $A_L$, identified with the vertex set $\operatorname{Vert}(L)$. We denote the Artin kernel by $BB_L^\varphi := \ker \varphi$. If $\varphi$ is the map sending each of the generators to $1 \in \Z$, then $BB_L^\varphi = BB_L$ is the usual Bestvina--Brady group.

Let $T_L$ be the Salvetti complex on $L$  and let $X_L$ be its universal cover. There is an  affine map $T_L \rightarrow S^1$ inducing $\varphi$ constructed as follows. For each $v \in \operatorname{Vert}(L)$, let $S^1_v := \R/\Z$ be the corresponding circle in $T_L$. Let $\sigma = \{v_1, \dots, v_k\} \in L$ be a simplex and let $T_\sigma = S_{v_1}^1 \times \cdots \times S_{v_k}^1$ be the associated subtorus of $T_L$. There is a map
\[
    T_\sigma \rightarrow S^1 = \R/\Z, \quad (x_1, \dots, x_k) \longmapsto \varphi(v_1) x_1 + \cdots + \varphi(v_k) x_k + \Z.
\]
The maps on each of the subtori extend to a well-defined map $f \colon T_L \rightarrow S^1$ inducing $\varphi$ on the level of fundamental groups. Moreover, $f$ induces a cube-wise affine $A_L$-equivariant height function $h \colon X_L \rightarrow \R$ making the diagram
\[\begin{tikzcd}
	X_L & \R \\
	{T_L} & {S^1}
	\arrow[from=1-1, to=2-1]
	\arrow["f", from=2-1, to=2-2]
	\arrow["h", from=1-1, to=1-2]
	\arrow[from=1-2, to=2-2]
\end{tikzcd}\]
commute --- here the vertical arrows are the universal covering maps.

We borrow the following definition and terminology from \cite{BuxGonzalez1999}.

\begin{defn}\label{def:livingdead}
    A vertex $v$ of $L$ is \textit{living} [resp., \textit{dead}] if $\varphi(v) \neq 0$ [resp., $\varphi(v) = 0$]. Denote the full subcomplex of $L$ spanned by the living [resp., dead] vertices by $L^\mathsf{a}$ [resp., $L^\mathsf{d}$].
\end{defn}

Let $Z = h\inv(\{p\})$ for some $p \notin \Z$. The level set $Z$ has a natural CW complex structure and $BB_L^\varphi$ acts cocompactly on $Z$; we denote the quotient $Z/BB_L^\varphi$ by $Y$. For each $n$-simplex $\sigma \in L$, the subtorus $T_\sigma \subseteq T_L$ lifts to $X_\sigma$, a collection of pairwise disjoint \textit{sheets} in $X_L$. Each sheet is an isometrically embedded copy of $(n+1)$-dimensional Euclidean space. 

Let $\mathsf{P}$ be the poset of simplices of $L$ that contain at least one vertex in $L^\mathsf{a}$. Then $Z$ is covered by the collection $\{ X_\sigma \cap Z \}_{\sigma \in \mathsf{P}}$. Writing $Y_\sigma$ for the image of $X_\sigma \cap Z$ in $Y$, we obtain a poset of spaces $\mathcal{Y} = \{ Y_\sigma \}_{\sigma \in \mathsf P}$ of $Y$ where each subcomplex $Y_\sigma$ is a disjoint union of tori or a set of vertices. Crucially, $\mathcal Y$ is a confident cover.

\begin{lem}\label{lem:countPoints}
    If $\sigma = \{v\} \in L^\mathsf{a}$ is a vertex, then $Y_\sigma$ is a collection of $\abs{\varphi(v)}$ vertices.
\end{lem}

\begin{proof}
    We will show that there are exactly $\abs{\varphi(v)}$ orbits of lines in $X_\sigma$ under the $BB_L^\varphi$-action on $X_L$. For each vertex $v$ in $L^\mathsf{a}$, evenly subdivide each edge of $X_v$ into $\abs{\varphi(v)}$ segments. Note that the restriction $\overline{X}_L^{(1)} \rightarrow \R$ of $h$ is cellular, where $\overline{X}_L^{(1)}$ is the subdivided $1$-skeleton of $X_L$ and $\R$ is given the cell structure with vertex set $\Z$ and edge set $\{[n,n+1] : n \in \Z\}$.

    Fix a vertex $\sigma = \{v\} \in L^\mathsf{a}$ and let $\overline{X}_\sigma$ be the subdivision of $X_\sigma$. Let $e$ be an edge of $\overline{X}_\sigma$ and let $e'$ be the unique edge of $X_\sigma$ such that $e \subseteq e'$. We say $e$ is an edge of \textit{type $i$} if it is the $i$th highest edge (under the height function $h$) contained in $e'$; the integer $i$ can take values in $\{0, \dots, \abs{\varphi(v)} -1 \}$.

    Because the action of $BB_L^\varphi$ on $X_L$ is height preserving, it preserves the set of type $i$ edges. Since $\varphi$ is surjective, $\gcd(( \varphi(v) )_{v \in L^{(0)}}) = 1$. Therefore, the generic level set $Z$ intersects edges of type $i$ for every $i \in \{0, \dots, \abs{\varphi(v)} - 1\}$. Moreover, $BB_L^\varphi$ acts transitively on the set of edges of type $i$ of the same height, which follows from the fact that $BB_L^\varphi$ acts transitively on the set of edges of $T_\sigma$ of the same height. Thus, we conclude that there are exactly $\abs{\varphi(v)}$ orbits of lines in $X_\sigma$ under the $BB_L^\varphi$ action. \qedhere
\end{proof}

\begin{thm}\label{thm:BBapprox}
    Let $\varphi \colon A_L \rightarrow \Z$ be an epimorphism and let $Y$ be a generic level set of the induced height function. If $\mathbb FBB_L^\varphi \rightarrow \mathcal D$ is an agrarian embedding, then
    \[
        b_p^\mathcal{D}(Y) = b_p^{(2)}(Y ; \mathbb F) = \sum_{v \in L^{(0)}} \abs{\varphi(v)} \cdot \widetilde{b}_{p-1}(\lk(v); \mathbb F).
    \]
    Moreover, if $BB_L^\varphi$ is of type $\mathsf{FP}_n(\mathbb F)$ then
    \[
        b_p^\mathcal D(BB_L^\varphi) = b_p^{(2)}(BB_L^\varphi; \mathbb F) = \sum_{v \in L^{(0)}} \abs{\varphi(v)} \cdot \widetilde{b}_{p-1}(\lk(v); \mathbb F).
    \]
    for all $p \leqslant n$. We also have $b_p^\mathcal D(H) = b_p^{(2)}(H; \mathbb F)$ whenever $H$ is a finite index subgroup of $BB_L^\varphi$ and $\mathcal D$ is the Hughes-free division ring of $\FF H$.
\end{thm}

Before beginning the proof, we remark that $BB_L^\varphi$ is a RFRS group, being a subgroup of a RAAG. Hence, by \cite[Corollary 1.3]{JaikinZapirain2020THEUO}, $\mathcal D_{\mathbb F BB_L^\varphi}$ exists. When $\mathbb F = \Q$, the Hughes free division ring and the Linnell skew field coincide, so in this case \cref{thm:BBapprox} computes the $\ell^2$-Betti numbers of $BB_L^\varphi$.  This generalises the computation of Davis and Okun in \cite[Theorem 4.4]{DavisOkun2012}

\begin{proof}
    The first statement is an immediate consequence of \cref{prop:nerveHomology,prop:nerveAgrarian}, \Cref{rem agrarian spaces}, \Cref{rem gradient spaces}, \cref{lem:countPoints}, and the observation that $\mathcal Y$ is a confident cover. The second follows from the fact that if $BB_L^\varphi$ is of type $\mathsf{FP}_n(\mathbb F)$ if and only if $Z$ is $n$-acyclic with $\mathbb{F}$ coefficients \cite{BuxGonzalez1999} (Bux--Gonzalez consider only the case $\mathbb F = \Z$, but their result remains true over arbitrary coefficients). \qedhere
\end{proof}

\begin{rem}\label{rem:livingDead}
    The condition that $BB_L^\varphi$ is of type $\mathsf{FP}_n(\mathbb F)$ can be verified directly in the flag complex $L$. Recall that a topological space is $n$-acyclic if its reduced homology (with coefficients in $\mathbb F$ in our case) vanishes in degrees $\leqslant n$. Note that we use the convention that the reduced homology of the empty set is $\mathbb F$ in dimension $-1$, so if $X$ is $n$-acyclic for $n \geqslant -1$, then $X$ is nonempty. 

    Bux and Gonzalez show that $BB_L^\varphi$ is of type $\mathsf{FP}_n(\mathbb F)$ if and only if $L^\mathsf a \cap \lk(\sigma)$ is $(n - \dim(\sigma) - 1)$-acyclic (with coefficients in $\mathbb F$) for every simplex $\sigma \in L^\mathsf d$, including the empty simplex which has dimension $-1$ and link $L$ \cite[Theorem 14]{BuxGonzalez1999}. Their result is stated in the case $\mathbb F = \mathbb \Z$ but remains true, with the same proof, when stated over a general coefficient ring. We will use this characterisation of the finiteness properties of $BB_L^\varphi$ is \cref{sec:fibring}.
\end{rem}

It is known by work of Okun--Schreve \cite{OkunSchreveTorsion} that for RAAGs the homology torsion growth $t_p^{(2)}(A_L)$ in degree $p$ is equal to $|H_{p-1}(L;\Z)_{\mathrm{tors}}|$.  We conjecture an analogous result for Artin kernels.

\begin{conj}
If $BB_L^\varphi$ is of type $\mathsf{FP}_n$, then
\[
    t_p^{(2)}(BB_L^\varphi) = \limsup_{n\to\infty}\frac{\log|H_p(G_n;\Z)_{\mathrm{tors}}|}{[BB_L^\varphi:G_n]} = \sum_{v \in L^{(0)}} \abs{\varphi(v)} \cdot |H_{p-1}(\lk(v); \Z)_{\mathrm{tors}}|
\]
for $p \leqslant  n$ and any residual chain $(G_n)$.
\end{conj}

\subsection{Graph products}

Let $K$ be a simplicial complex on the vertex set $[m] := \{1, \dots, m\}$. Let $(\mathbf{X}, \mathbf{A}) = \{(X_i, A_i) : i \in [m] \}$ be a collection of CW-pairs. The \textit{polyhedral product} of $(\mathbf{X}, \mathbf{A})$ and $K$ is the space
\[
    (\mathbf{X}, \mathbf{A})^K := \bigcup_{\sigma \in K} \prod_{i = 1}^m Y_i^\sigma \quad \text{where} \quad Y_i^\sigma = 
    \begin{cases}
        X_i & \text{if } i \in \sigma, \\
        A_i & \text{if } i \notin \sigma .
    \end{cases}
\]

If $\mathbf{\Gamma} = \{\Gamma_1, \dots, \Gamma_m\}$ is a finite set of groups, then the \textit{graph product} of $\mathbf{\Gamma}$ and $K$, denoted $\mathbf{\Gamma}^K$, is the quotient of the free product $\ast_{i \in [m]} \Gamma_i$ by all the relations $[\gamma_i, \gamma_j] = 1$, where $\gamma_i \in \Gamma_i$, $\gamma_j \in \Gamma_j$ and $i$ and $j$ are adjacent vertices of $K$. Note that $\mathbf{\Gamma}^K$ is the fundamental group of $X = (B\mathbf{\Gamma}, \ast)^K$, where $B\mathbf{\Gamma} = \{ B\Gamma_1, \dots, B\Gamma_m \}$ and $\ast$ is a set of one-point subcomplexes. Note that by \cite[Theorem 1.1]{Stafa2015}, if $K$ is a flag complex and each $\Gamma_i$ is a discrete group, then $(B\mathbf{\Gamma}, \ast)^K$ is a model for a $K(\mathbf{\Gamma}^K,1)$.

\begin{thm}\label{thm:agrGraphProd}
    Let $\mathbb F$ and $\mathcal{D}$ be skew fields and let $K$ be a finite simplicial flag complex. Let $G = \mathbf{\Gamma}^K$ be a graph product of discrete groups such that there is an agrarian map $\mathbb FG \to \mathcal{D}$. If $H_\bullet(\Gamma; \mathcal{D}) = 0$ for each $\Gamma \in \mathbf{\Gamma}$, then
    \[
        H_p^\mathcal D(\mathbf \Gamma) \cong \widetilde{H}_{p-1}(K; \mathcal{D}).
    \]
    In particular, $b_p(G;\mathcal{D}) = b_{p-1}(K; \mathbb F)$.
\end{thm}

\begin{proof}
    Let $X = (B\mathbf{\Gamma}, \ast)^K$ and let $\mathsf{P}$ be the poset whose elements are (possibly empty) simplices of $K$. Note that $\mathsf{P}$ defines a poset of spaces $\{X_\sigma\}_{\sigma \in \mathsf P}$, where $X_\varnothing$ is a single vertex. For each $\varnothing \neq \sigma \in \mathsf{P}$, the group $\mathbf{\Gamma}^J$ is a direct product of groups with vanishing $\mathcal{D}$-homology, and therefore $H^n(\mathbf{\Gamma}^J; \mathcal{D}) = H^n(X_J; \mathcal{D}) = 0$ by the K\"unneth formula. Take the spectral sequence of \cref{lem:spec} with the coefficient system $\sigma \mapsto H_q(X_{\min(\sigma)}; \mathcal{D})$. All of these coefficient system cohomology groups vanish except when $\sigma = \varnothing$ and $q = 0$ and therefore
    \[
        H^p(X; \mathcal{D}) = H^p(\Flag(\mathsf{P}), \Flag(\mathsf{P}>\varnothing); \mathcal{D}) \cong \widetilde{H}^{p-1}(K; \mathcal{D}).
    \]
    by \cref{lem:spec}. The last isomorphism follows from the fact that $\mathsf{P}$ is isomorphic to the cone on the barycentric subdicision of $K$, where the cone point corresponds to the empty simplex $\varnothing$. \qedhere
\end{proof}

\begin{rem}
    There are certain situations in which one can easily deduce the fact that a graph product has an agrarian embedding. For example, if each group $\Gamma \in \mathbf \Gamma^K$ is ordered, then so is $\mathbf \Gamma^K$ by a result of Chiswell \cite{Chiswell2012}. Thus, $\mathbb F\mathbf \Gamma^K$ embeds into its Mal'cev--Neumann completion \cite{Malcev1948,Neumann1949}.

    Similarly, if each $\Gamma \in \mathbf \Gamma^K$ is RFRS, then it is possible to show that $\mathbf \Gamma^K$ is also RFRS and therefore has a Hughes-free division ring $\mathcal D_{\mathbb F\mathbf \Gamma^K}$. We provide a sketch of an argument here, which proceeds by induction on the number of vertices in $K$. If $K$ has one vertex, then the claim is trivial. Suppose now that $K$ has more than one vertex. If the $1$-skeleton of $K$ is a complete graph, then the claim is again trivial since $\mathbf \Gamma^K$ is a direct product of RFRS groups and therefore RFRS. If $K$ is not a complete graph, then there is some vertex $v \in K$ such that $\st(v) \neq K$ and we obtain a splitting $\mathbf\Gamma^K \cong \mathbf \Gamma^{\st(v)} \ast_{\mathbf\Gamma^{\lk(v)}} \mathbf\Gamma^{K \smallsetminus v}$. By induction, both $\mathbf \Gamma^{\st(v)}$ and $\mathbf\Gamma^{K \smallsetminus v}$ are RFRS and by a result of Koberda and Suciu, the amalgam is also RFRS \cite[Theorem 1.3]{KoberdaSuciu2020}. Note that the Koberda--Suciu result is a combination theorem for a related class of groups called RFR$p$ groups, but their proof is easily adapted to the RFRS case; we refer the interested reader to their paper for more details.
\end{rem}

Thanks to the work of Okun and Schreve \cite{OkunSchreveTorsion}, we have the following corollary which holds, in particular, for RAAGs.  

\begin{cor}
Let $\mathbb F$ and $\mathcal D$ be skew fields and let $K$ be a finite simplicial flag complex. Let $G = \mathbf{\Gamma}^K$ be a graph product with an agrarian map $\mathbb FG \to \mathcal D$. If each $\Gamma \in \mathbf \Gamma$ is residually finite, $\mathcal D$-acyclic, and $\mathbb F$-$\ell^2$-acyclic, then
\[
    b_n^\mathcal{D}(G) = \lim_{i \to \infty} \frac{b_n(G_i; \mathbb{F}_p)}{[G : G_i]}
\]
for any residual chain $\Gamma = \Gamma_0 \trianglerighteqslant G_1 \trianglerighteqslant G_2 \trianglerighteqslant \cdots$ of finite index normal subgroups.
\end{cor}

\begin{proof}
Okun and Schreve \cite[Theorem 5.1]{OkunSchreveTorsion} showed that the right-hand side of the above equation is independent of the choice of residual chain and equal to $\widetilde{b}_{p-1}(K;\mathbb F)$, which equals the left-hand side by \cref{thm:agrGraphProd}. \qedhere
\end{proof}

\begin{rem}
    Let $K$ be a simplicial complex, $X=(B\mathbf{\Gamma},\ast)^K$, and let $\widetilde{X}$ denote the universal cover.  The above arguments apply equally well for computing $b^\mc D_p(\widetilde{X})$ and $\btwo{p}(X;\FF)$.  In both cases they are equal to $\widetilde{b}_{p-1}(K;\FF)$ whenever they are defined.  In the case when $K$ is not a flag complex, $X$ is not aspherical.  
\end{rem}

\subsection{Artin groups}
Let $A$ be a residually finite Artin group and suppose the $K(A,1)$ conjecture holds for $A$.  Then there is a contractible complex $D_A$ --- the \emph{Deligne complex of $A$} --- with stabilisers the maximal parabolic subgroups of $A$ admitting a strict fundamental domain $Q_A$.  In \cite[Section~4 and Theorem~5.2]{OkunSchreveTorsion} the authors compute $b^{(2)}_p(A;\FF)=\tilde b_{p-1}(\partial Q_A;\FF)$ independently of a choice of residual chain.  Here $\partial Q$ is the subcomplex of $Q$ with non-trivial stabilisers.

\begin{thm}\label{thm Artin groups}
    Let $\FF$ be a field.  Let $A$ be a residually finite Artin group.  Suppose the $K(A,1)$ conjecture holds for $A$.  If $\FF A\to \mathcal D$ is an agrarian map, then
    \[b^{(2)}_p(A;\FF)=b^{\mathcal D}_p(A;\FF)=\tilde b_{p-1}(\partial Q_A;\FF). \]
\end{thm}
\begin{proof}
In the action of $A$ on $D_A$, the non-trivial stabilisers have a central $\Z$ subgroup and so they are $\mc D$-acyclic and have vanishing $\FF$-homology growth.  We take a poset of spaces $\mc X$ over the (barycentric subdivision of the) strict fundamental domain $Q$ of $A$, where we assign a classifying space $BA_\sigma$ to each $\sigma\in Q$.  Now, we apply \Cref{rem agrarian spaces} and \Cref{rem gradient spaces}.
\end{proof}

\begin{rem}\label{sec strict fund}
The above argument applies to residually finite groups $G$ acting on a contractible complex with strict fundamental domain and $\mc D$-acyclic stabilisers---whenever $G$ admits an agrarian embedding $\FF G\to \mc D$.
\end{rem}

\subsection{Complements of hyperplane arrangements}
Let $\cala$ be a collection of affine hyperplanes in $\C^k$ and let $\Sigma(\cala)$ denote their union.  Let $M(\cala)\coloneqq \C^k \smallsetminus \Sigma(\cala)$.  The \emph{rank} of $M(\cala)$ is the maximum codimension $n$ of any nonempty intersection of hyperplanes in $\cala$.  By \cite[Proposition~2.1]{DavisJanuszkiewiczLeary2007} the ordinary Betti numbers satisfy  $\widetilde{b}_p(M(A);\FF)=0$ except possibly when $p=n$.

\begin{thm}\label{hyperplanes}
    Let $\FF$ be a skew field, $\cala$ be an affine hyperplane arrangement in $\C^k$ of rank $n$, and let $\Gamma\coloneqq\pi_1M(A)$. 
    \begin{enumerate}
        \item If $\Gamma$ is residually finite, then $b^{(2)}_p(M(\cala);\FF)=\widetilde{b}_p(M(\cala);\FF)$ which equals zero except possibly when $p=n$.
        \item If $\alpha\colon\FF\Gamma\to\mc D$ is an agrarian map, then
        $b^\mc D_p(M(\cala))=\widetilde{b}_p(M(\cala);\FF)$ which equals zero except possibly when $p=n$.
    \end{enumerate} 
\end{thm}

We will need a lemma.

\begin{lem}\label{lem central hyperplanes}
Let $\FF$ be a skew field, $\cala$ be a non-empty central affine hyperplane arrangement in $\C^k$ of rank $n$, and let $\Gamma\coloneqq\pi_1M(A)$. 
    \begin{enumerate}
        \item If $\Gamma$ is residually finite, then $b^{(2)}_p(M(\cala);\FF)=0$ for all $n$ independently of a choice of residual chain.
        \item If $\alpha\colon\FF\Gamma\to\mc D$ is an agrarian map, then
        $b^\mc D_p(M(\cala))=0$ for all $n$.
    \end{enumerate} 
\end{lem}
\begin{proof}
    In this case we have $M(\cala)=S^1\times B$ where $B=M(\cala)/S^1$ by \cite[Proof of Lemma~5.2]{DavisJanuszkiewiczLeary2007}.  Both results are easy applications of the K\"unneth formula and vanishing of the relevant invariant for $\Z=\pi_1S^1$. We spell out the details in the first case to highlight the independence of the residual chain.
    
    Observe that every finite cover $M_i$ of $M(\cala)$ can be written
    as $B_i\times S^1$ such that $S^1$ has $m$ one cells.  We have the index of the cover $|M:M_i|=m|B:B_i|$.  Now, we compute via the K\"unneth formula that
    \begin{align*}
    b^{(2)}_p(M,(M_i);\FF)&=\lim_{i\to\infty}\frac{b_p(M_i)+b_{p-1}(M_i)}{|M:M_i|}\\
    &=\lim_{i,m\to\infty}\frac{b_p(M_i)+b_{p-1}(M_i)}{m|B:B_i|}\\
    &=0.\qedhere
    \end{align*}
\end{proof}

\begin{proof}[Proof of \Cref{hyperplanes}]
    By \cite[Section~3]{DavisJanuszkiewiczLeary2007} there is a cover $\mc U$ of $M(\cala)$ by central subarrangements $U_\sigma$ such that $\pi_1 U_\sigma \to \Gamma$ is injective and the nerve $N(\mc U)$ is contractible.  There is also a cover of a deleted neighbourhood of $\Sigma(A)$, denoted $\mc U_{\mathrm{sing}}$, such that $H_p(N(\mc U),N(\mc U_{\mathrm{sing}}))$ is concentrated in degree $n$.  It follows from \Cref{lem central hyperplanes} that either the homology gradients or the $\mc D$-Betti numbers vanish.  In particular, by Remarks~\ref{rem gradient spaces} and \ref{rem agrarian spaces} the cover $\mc U$ is confident and the results follow.
\end{proof}

\section{A lower bound for homology gradients} \label{sec:kaz}

Let $R$ be a ring, let $\mathcal D$ be a skew-field, and let $\varphi \colon R \rightarrow \mathcal D$ be a ring homomorphism. There is a rank function $\rk_{\mathcal D, \varphi} \colon \mathrm{Mat}(R) \rightarrow \R_{\geqslant 0}$ defined by $\rk_{\mathcal D, \varphi} A = \rk \varphi_* A$, where $\varphi_* A$ is the matrix obtained by applying the homomorphism $\varphi$ to every entry of $A$ and $\rk \varphi_* A$ is the rank of $\varphi_* A$ as a matrix over $\mathcal D$. If $\varphi \colon R \hookrightarrow \mathcal D$ is an epic embedding (i.e.~$\varphi(R)$ generates $\mathcal D$ as a skew-field) and $\rk_{\mathcal D, \varphi} \geqslant \rk_{\mathcal E, \psi}$ for every skew-field $\mathcal E$ and every ring homomorphism $\psi \colon R \rightarrow \mathcal E$, then $\mathcal D$ is said to be a \textit{universal division ring of fractions} for $R$. If $\mathcal D$ is a universal division ring of fractions for $R$, it is then unique up to $R$-isomorphism \cite[Theorem 4.4.1]{cohn1995skew}.

\begin{thm}[Jaikin-Zapirain {\cite[Corollary 1.3]{JaikinZapirain2020THEUO}}] \label{thm:jaikinUniv}
    Let $G$ be a residually (locally indicable and amenable) group and let $\mathbb F$ be a skew-field. Then the Hughes-free division ring $\DF{G}$ exists and is the universal division ring of fractions of $\mathbb FG$.
\end{thm}

We note that \cref{thm:jaikinUniv} holds for RFRS groups, since they are residually poly-$\Z$ (see, e.g., \cite[Proposition 4.4]{JaikinZapirain2020THEUO}). The main theorem of this section will follow quickly from the observation that $\DF{G}$-Betti numbers scale under taking finite index subgroups and another result of Jaikin-Zapirain (\cite[Corollary 1.6]{JaikinZapirain2020THEUO}), which he states for $\ell^2$-Betti numbers of CW complexes but also holds for agrarian homology. For the convenience of the reader, we reproduce a proof in the agrarian setting here.

\begin{thm}[Jaikin-Zapirain] \label{thm:jaikinRank}
    Let $\mathbb F$ be a skew-field and suppose that $G$ is a group of type $\FP_{n+1}(\mathbb F)$ for some $n \in \N$ such that $\DF{G}$ exists and is the universal division ring of fractions of $\mathbb FG$. Then $b_m^{\DF{G}}(G) \leqslant b_m(G; \mathbb F)$ for all $m \leqslant n$.
\end{thm}

\begin{proof}
    The embedding $\iota \colon \mathbb FG \hookrightarrow \DF{G}$ and the augmentation map $\alpha \colon \mathbb FG \rightarrow \mathbb F$ induce rank functions on $\mathrm{Mat}(\mathbb FG)$ which we denote by $\rk_G$ and $\rk_{\mathbb F}$, respectively. By universality, $\rk_G \geqslant \rk_\mathbb F$. 
    
    Let $C_\bullet \rightarrow \mathbb F$ be a free-resolution of the trivial $\mathbb FG$-module $\mathbb F$ such that $C_m$ is finitely generated for all $m \leqslant n+1$. For all $m \leqslant n+1$, let $d_m$ be an integer such that $C_m \cong \mathbb FG^{d_m}$ and view the boundary maps $\partial_m \colon C_m \rightarrow C_{m-1}$ as matrices over $\mathbb FG$. The homologies we are interested in are $H_m(\DF{G} \otimes_{\mathbb FG} C_\bullet)$ and $H_m(k \otimes_{\mathbb FG} C_\bullet)$ and note that the differentials $\DF{G} \otimes \partial_m$ and $k \otimes \partial_m$ correspond to the matrices $\iota_* \partial_m$ and $\alpha_* \partial_m$ under the identifications $\DF{G} \otimes_{\mathbb FG} C_m \cong \DF{G}^{d_m}$ and $\mathbb F \otimes_{\mathbb FG} C_m \cong \mathbb F^{d_m}$. Therefore
    \begin{align*}
        b_m^{\DF{G}}(G) &= d_m - \rk_G \partial_m - \rk_G \partial_{m+1} \\ &\leqslant d_m - \rk_{\mathbb F} \partial_m - \rk_{\mathbb F} \partial_{m+1} \\
        &= b_m(G; \mathbb F) \qedhere
    \end{align*}
\end{proof}

As a consequence, we obtain that agrarian Betti numbers bound homology gradients from below.

\begin{thm}\label{thm:kaz}
    Let $\mathbb F$ be a skew-field and let $G$ be a group of type $\FP_{n+1}(\mathbb F)$ such that $\DF{G}$ exists and is the universal division ring of fractions of $\mathbb FG$. If $H \leqslant G$ is any subgroup of finite index, then
    \[
        b_m^{\DF{G}} (G) \leqslant \frac{b_m(H; \mathbb F)}{[G:H]}
    \]
    for all $m \leqslant n$. In particular, if $G$ is residually finite and $(G_i)_{i \in \N}$ is a residual chain of finite-index subgroups, then $b_m^{\DF{G}} (G) \leqslant b_m^{(2)}(G, (G_i); \mathbb F)$
    for all $m \leqslant n$.
\end{thm}

\begin{proof}
    By \cite[Lemma 6.3]{Fisher2021improved}, $[G:H] \cdot b_m^{\DF{G}}(G) = b_m^{\DF{H}}(H)$, and by \cref{thm:jaikinRank}. The second claim is an immediate consequence of the first. \qedhere
\end{proof}

\section{Applications to fibring} \label{sec:fibring}

A group $G$ is \textit{algebraically fibred} (or, simply, \textit{fibred}) if it admits a nontrivial homomorphism $G \rightarrow \Z$ with finitely generated kernel. More generally, if $\mathcal P$ is a finiteness property (e.g.~type $\F_n$ or $\FP_n(R)$ for some ring $R$, see \cref{subsec:finprops} for definitions), we say that $G$ is \textit{$\mathcal P$-fibred} if there is a nontrivial homomorphism $G \rightarrow \Z$ with kernel of type $\mathcal P$.

\begin{thm}\label{thm no fibring}
    Let $\mathbb F$ be a skew-field and let $G$ be a RFRS group of type $\FP_{n+1}(\mathbb F)$.  If $G$ is not virtually $\FP_n(\FF)$-fibred, then for every residual chain of finite-index subgroups $(G_i)_{i \in \N}$, we have $b_m^{(2)}(G, (G_i); \mathbb F)>0$
    for some $m \leqslant n$.
\end{thm}
\begin{proof}
    By \cite{Fisher2021improved} a RFRS group $G$ is virtually $\FP_n(\mathbb{F})$-fibred if and only if $b_i^{\mathcal{D}_{\mathbb{F}G}}(G) = 0$ for every $i \leqslant n$. 
    Since $G$ is not virtually $\FP_n(\FF)$-fibred,  we have $b_m^{\DF{G}}(G)>0$ for some $m\leqslant n$.  The result now follows from \Cref{thm:kaz}.
\end{proof}

The authors thank Robert Kropholler for communicating to us the following question due to Matthew Zaremsky: If a RAAG $A_L$ is virtually $\F_n$-fibred, is it $\F_n$-fibred? We are able to answer the analogous homological question over skew fields, $\Z$, and $\Z/m$ for $m \in \N_{>1}$.

\begin{thm}\label{thm:virtFibiffFib}
    Let $L$ be a finite flag complex and $R$ be either a skew field $\mathbb F$, $\Z$, or $\Z/m$ for some $m \in \N_{>1}$. Then the RAAG $A_L$ is virtually $\FP_n(R)$-fibred if and only if it is $\FP_n(R)$-fibred.
\end{thm}

\begin{proof}
    We begin with the case $R = \mathbb F$. Let $A_L$ be virtually $\FP_n(\mathbb F)$-fibred. By \cite{Fisher2021improved},  a RFRS group $G$ is virtually $\FP_n(\mathbb{F})$-fibred if and only if $b_i^{\mathcal{D}_{\mathbb{F}G}}(G) = 0$ for every $i \leqslant n$.  In particular this applies in the case $G=A_L$.  By \cref{thm:agrGraphProd}, 
    $b_i^{\mathcal{D}_{\mathbb{F}A_L}}(A_L)=0$ for every $i\leqslant n$ implies that $\widetilde{b}_i(L; \mathbb F) = 0$ for every $i \leqslant n-1$. By \cite[Main Theorem]{BestvinaBrady}, the Bestvina--Brady group is of type $\FP_n(\mathbb{F})$ and therefore $A_L$ is $\FP_n(\mathbb F)$-fibred.

    Now suppose $R = \Z/m$ for some $m \in \N_{>1}$ and suppose that $A_L$ is virtually $\FP_n(\Z/m)$-fibred. If $p$ is a prime factor of $m$, then there is a ring homomorphism $\Z/m \rightarrow \Z/p = \mathbb F_p$ and therefore $\mathbb F_p$ is a $\Z/m$-algebra. Thus, $A_L$ is virtually $\FP_n(\mathbb F_p)$-fibred. Therefore $b_i^{\mathcal D_{\mathbb F_p A_L}}(A_L) = 0$ for all $i \leqslant n$ and thus $\widetilde{b}_i(L; \mathbb F_p) = 0$ for all $i \leqslant n-1$ for all $i \leqslant n-1$ by \cref{thm:agrGraphProd}. Then $\widetilde{b}_i(L; \Z/m) = 0$ for all $i \leqslant n-1$, so $BB_L$ is of type $\FP_n(\Z/m)$.
    
    Finally, suppose $A_L$ is virtually $\FP_n$-fibred. In particular, $A_L$ is virtually $\FP_n(\mathbb F)$-fibred for every skew field $\mathbb F$, which implies that $L$ is $(n-1)$-acyclic over every field by \cref{thm:agrGraphProd}. Therefore $L$ is $(n-1)$-acyclic over $\Z$, which implies that $BB_L$ is of type $\FP_n$. \qedhere
\end{proof}

\begin{rem}
    In the case $\mathbb F = \mathbb Q$, \cref{thm:virtFibiffFib} could have been deduced from previous work since the $\ell^2$-Betti numbers of RAAGs were computed by Davis and Leary in \cite{DavisLeary2003} and it is well known that a virtual $\FP_n(\Q)$-fibring implies the vanishing of $\ell^2$-Betti numbers in dimensions $\leqslant n$.
\end{rem}

Since a finitely presented group of type $\mathsf{FP_n}$ is of type $\mathsf{F}_n$ we can reduce Zaremsky's question to one remaining case.

\begin{q}[Zaremsky]
Let $L$ be a finite flag complex.  If $A_L$ is virtually $\mathsf{F}_2$-fibred, then is it $\mathsf{F}_2$-fibred?
\end{q}

We can also give examples of RAAGs that show that \cite[Theorem A]{Fisher2021improved} does not hold when $\Q$ is replaced by $\Z$. In other words, the vanishing of $\ell^2$-Betti numbers of RFRS groups does not detect virtual $\FP_n$-fibrations.

\begin{prop}\label{prop:noZfibre}
    Let $p$ be a prime. There are RAAGs that are $\ell^2$-acyclic but that do not virtually $\FP_2(\mathbb F_p)$-fibre. In particular, these RAAGs do not virtually $\FP_2$-fibre.
\end{prop}

\begin{proof}
    Let $L$ be a $\Q$-acyclic flag complex that has non-trivial $\mathbb F_p$-homology in dimension $1$, e.g. we may take $L$ to be a flag triangulation of $M(1,p)$ the Moore space with homology $\widetilde{H}_n(M(1,p);\Z)=0$ unless $n=1$, in which case it is isomorphic to $\Z/p$. Then $A_L$ is $\ell^2$-acyclic by \cite{DavisLeary2003} (or \cref{thm:agrGraphProd}) but it is not $\mathcal{D}_{\mathbb F_p A_L}$-acyclic by \cref{thm:agrGraphProd}. By \cite[Theorem 6.6]{Fisher2021improved}, $A_L$ does not virtually $\FP_2(\mathbb F_p)$-fibre and in particular does not virtually $\FP_n$-fibre.
\end{proof}

In contrast to this result, Kielak showed that if $G$ is RFRS, of cohomological dimension at most $2$, and $\ell^2$-acyclic, then $G$ is virtually $\FP_2$-fibred \cite[Theorem 5.4]{KielakRFRS}.

The next application has to do with the following general question: If $G$ fibres in two different ways, so that $G \cong K_1 \rtimes \Z \cong K_2 \rtimes \Z$ with $K_1$ and $K_2$ finitely generated, then what properties do $K_1$ and $K_2$ share? For example if $G$ is a free-by-cyclic (resp.~surface-by-cyclic) group and $G \cong K \rtimes \Z$, then $K$ is necessarily a free (resp.~surface) group. We thank Ismael Morales for bringing the following question to our attention: if $G \cong K_1 \rtimes \Z \cong K_2 \rtimes \Z$ with $K_1$ and $K_2$ finitely generated, then is $b_1^{(2)}(K_1) = 0$ if and only if $b_1^{(2)}(K_2) = 0$? We prove this is the case for RAAGs, and obtain a similar result for higher $\ell^2$-Betti numbers and other agrarian invariants.

\begin{thm}\label{thm:fibresFibre}
    Let $\varphi_0, \varphi_1 \colon A_L \rightarrow \Z$ be epimorphisms such that $BB_L^{\varphi_0}$ and $BB_L^{\varphi_1}$ are of type $\mathsf{FP}_n(\mathbb F)$. If $\mathbb F BB_L^{\varphi_i} \hookrightarrow \mathcal D_i$ is an agrarian embedding for $i = 0,1$, then $BB_L^{\varphi_0}$ is $\mathcal D_0$-acyclic in dimensions $\leqslant n$ if and only if $BB_L^{\varphi_1}$ is $\mathcal D_1$-acyclic in dimensions $\leqslant n$.
\end{thm}

Before proving \cref{thm:fibresFibre}, we need a technical lemma. First we fix some notation. If $\sigma_1 = [e_1, \dots, e_m]$ and $\sigma_2 = [f_1, \dots, f_m]$ are ordered simplices in a simplicial complex $L$ such that $\sigma_1 \cup \sigma_2$ is a simplex (equivalently, if $\sigma_1 \in \lk(\sigma_2)$), then $\sigma_1 \cup \sigma_2$ always denotes the \textit{ordered} simplex $[e_1, \dots, e_m, f_1, \dots, f_n]$. Moreover, if $\tau = \alpha_1 \sigma_1 + \cdots + \alpha_n \sigma_n$ is a formal linear combination of simplices $\sigma_i$ (with coefficients $\alpha_i$ in some fixed skew field) such that $\sigma \cup \sigma_i \in L$ for every $i$, then $\sigma \cup \tau$ denotes the formal linear combination
\[
    \alpha_1 \sigma \cup \sigma_1 + \cdots + \alpha_n \sigma \cup \sigma_n.
\]
If $\varphi \colon A_L \rightarrow \Z$ is a homomorphism, recall that $L^\mathsf{a}$ is the subcomplex of $L$ spanned by the vertices $v \in L$ such that $\varphi(v) \neq 0$. We will write $\lk_L(\sigma)$ (resp.~$\lk_{L^\mathsf{a}}(\sigma)$) for the link of a simplex $\sigma$ in $L$ (resp.~$L^\mathsf{a}$).

\begin{lem}\label{lem:Living}
    Let $BB_L^\varphi$ be of type $\mathsf{FP}_n(\mathbb F)$ and let $v$ be a dead vertex of $L$. Then every simplicial $(n-1)$-cycle of $\operatorname{lk}_{L}(v)$ is homologous to a cycle in $\operatorname{lk}_{L^\mathsf a}(v)$.
\end{lem}

\begin{proof}
    Let $\sigma = \alpha_1 \sigma_1 + \cdots + \alpha_k \sigma_k$ be a simplicial $(n-1)$-cycle in $\operatorname{lk}_L(v)$, where each $\sigma_i$ is an ordered $(n-1)$-simplex of $\operatorname{lk}_L(v)$ and $\alpha_i \in \mathbb F$ for each $i$.  By induction on $m\geq 0$, we will show that the simplices $\sigma_i$ can be replaced with $(n-1)$-simplices having at least $m$ living vertices such that the resulting chain is a cycle homologous to $\sigma$. The lemma follows from the $m = n$ case.

    For the base case, suppose that $\sigma_i$ is a simplex with no living vertices. Then $\{v\} \cup \sigma_i$ is a dead $n$-simplex and therefore $\operatorname{lk}_{L^\mathsf a}(\{v\} \cup \sigma_i)$ is $(-1)$-connected (see \cref{rem:livingDead}), i.e.~it is nonempty. Thus, there is a living vertex $u$ such that $\{u\} \cup \sigma_i \subseteq \operatorname{lk}_L(v)$. Since
    \[
        \partial ( \{u\} \cup \sigma_i ) = \sigma_i - \{u\} \cup \partial \sigma_i,
    \]
    where $\{u\} \cup \partial \sigma_i$ is a linear combination of $(n-1)$-simplices with one living vertex, we can replace $\sigma_i$ with $\{u\} \cup \partial \sigma_i$ in $\sigma$. Hence, we assume that the linear combination $\alpha_1 \sigma_1 + \cdots + \alpha_k \sigma_k$ only involves simplices with at least one living vertex.

    Assume that $\alpha_1 \sigma_1 + \cdots + \alpha_k \sigma_k$ only involves simplices with at least $m \geqslant 1$ living vertices for some $m<n$ and let $\sigma_i$ be a simplex with exactly $m$ living vertices. Let $\lambda \subseteq \sigma_i$ be the dead $(n-m-1)$-face of $\sigma_i$ and let $\sigma_i = \sigma_{i_1}, \dots, \sigma_{i_l}$ be the simplices among $\{\sigma_1, \dots, \sigma_k \}$ containing $\lambda$ as a face. For each $j \in \{1, \dots, l\}$, write $\sigma_{i_j} = \varepsilon_j \lambda \cup \tau_j$, where $\tau_j$ is a living $(m-1)$-simplex of $\operatorname{lk}_L(v)$ and $\varepsilon_j \in \{\pm 1\}$. Then
    \begin{align*}
        \partial \left( \sum_{j=1}^l \alpha_{i_j} \sigma_{i_j} \right) &= \sum_{j=1}^l \alpha_{i_j} \varepsilon_j (\partial \lambda \cup \tau_j + (-1)^{n-m} \lambda \cup \partial \tau_j) \\
        &= \left( \sum_{j=1}^l \alpha_{i_j} \varepsilon_j \partial \lambda \cup \tau_j \right) + (-1)^{n-m} \lambda \cup \partial \left( \sum_{j=1}^l \alpha_{i_j} \varepsilon_j \tau_j \right) \\
        &= 0,
    \end{align*}
    since $\partial \sigma = 0$ and the simplices $\sigma_{i_j}$ are the only simplices among $\{\sigma_1, \dots, \sigma_k \}$ containing $\lambda$ as a face. Thus, 
    \[
        \lambda \cup \partial \left( \sum_{j=1}^l \alpha_{i_j} \varepsilon_j \tau_j \right) = 0,
    \]
    whence we conclude that $\sum_{j=1}^l \alpha_{i_j} \varepsilon_j \tau_j$ is an $(m-1)$-cycle in $\operatorname{lk}_{L^\mathsf a}(\{v\} \cup \lambda)$. But $\{v\} \cup \lambda$ is a dead $(n-m)$-simplex and therefore $\operatorname{lk}_{L^\mathsf a}(\{v\} \cup \lambda)$ is $(m-1)$-connected. Hence, $\sum_{j=1}^l \alpha_{i_j} \varepsilon_j \tau_j = \partial \psi$ for some living $m$-chain $\psi$ in $\operatorname{lk}_{L^\mathsf a}(\{v\} \cup \lambda)$. Then
    \begin{align*}
        \partial(\lambda \cup \psi) &= \partial \lambda \cup \psi + (-1)^{n-m} \lambda \cup \left( \sum_{j=1}^l \alpha_{i_j} \varepsilon_j \tau_j \right) \\
            &= \partial \lambda \cup \psi + (-1)^{n-m} \sum_{j=1}^l \alpha_{i_j} \sigma_{i_j}.
    \end{align*}
    The chain $\partial \lambda \cup \psi$ is a linear combination of simplices with $m + 1$ living vertices. We can therefore replace $\sum_{j=1}^l \alpha_{i_j} \sigma_{i_j}$ with $\pm \partial \lambda \cup \psi$ and assume that $\sigma$ is a linear combination of simplices each with at least $m+1$ living vertices. \qedhere
\end{proof}

\begin{proof}[Proof (of \cref{thm:fibresFibre})]
    Suppose that $b_p^{\mathcal D_0}(BB_L^{\varphi_0}) > 0$ for some $p \leqslant n$. By \cref{thm:BBapprox}, there is a vertex $v$ of $L$ such that $\varphi_0(v) \neq 0$ and 
    \[
        \widetilde{b}_{p-1}(\operatorname{lk}(v); \mathbb F) > 0.
    \]
    Hence, there is a simplicial $(p-1)$-cycle $\sigma$ in $\operatorname{lk}(v)$ that is not a boundary. If $\varphi_1(v) = 0$, then, by \cref{lem:Living}, $\sigma$ is homologous to a cycle in $\operatorname{lk}_{L^\mathsf{a}}(v)$ where $L^\mathsf a$ denotes the living link with respect to $\varphi_1$. Thus $\widetilde H_{p-1}(\operatorname{lk}_{L^\mathsf a}(v); \mathbb F) \neq 0$. But $\operatorname{lk}_{L^\mathsf a}(v)$ is $(n-1)$-connected over $\mathbb{F}$, so we must have $\varphi_1(v) \neq 0$, and therefore $b_p^{\mathcal D_1}(BB_L^{\varphi_1}) > 0$ by \cref{thm:BBapprox}. \qedhere
\end{proof}

We highlight the following immediate corollary.

\begin{cor}\label{cor:fibresFibre}
    Either all the $\mathsf{FP}_n(\mathbb F)$-fibres of $A_L$ are virtually $\mathsf{FP}_n(\mathbb F)$-fibred or none of them are. In particular, either all of $A_L$'s fibres virtually fibre or none of them do.
\end{cor}

\begin{proof}
    This follows from \cref{thm:fibresFibre} and the fact that being $\mathcal{D}_{\mathbb F BB_L^{\varphi}}$-acyclic in dimensions $\leqslant n$ is equivalent to virtually fibring with kernel of type $\mathsf{FP}_n(\mathbb F)$ \cite[Theorem 6.6]{Fisher2021improved}. \qedhere 
\end{proof}

\section{Amenable category and minimal volume entropy}\label{sec min vol entropy}
In this section we will relate $\DF{G}$-Betti numbers with amenable category and minimal volume entropy.

Let $X$ be a path-connected space with fundamental group $G$.  A (not necessarily path-connected) subset $U$ of $X$ is an \emph{amenable subspace} if $\pi_1(U,x)\to \pi(X,x)$ has amenable image for all $x\in U$.  The \emph{amenable category}, denoted $\mathsf{cat}_{\mc{AMN}}X$, is the minimal $n\in \N$ for which there exists an open cover of $X$ by $n+1$ amenable subspaces.  If no such cover exists we set $\mathsf{cat}_{\mc{AMN}}X=\infty$.  The definition of amenable category has been extracted from \cite{CapovillaLohMoraschini2022} and \cite{Li2022amenable}.  Note that we normalise the invariant as in the second paper.  Also note that often in the literature the multiplicity of the cover is considered instead, however, the two definitions turn out to be equivalent for CW complexes \cite[Remark~3.13]{CapovillaLohMoraschini2022}

\begin{prop}\label{AMN cat vs DFG}
Let $\FF$ be a skew field.  Let $G$ be residually finite of type $\mathsf{F}$, and suppose $\DF{G}$ exists. If $\mathsf{cat}_{\mc{AMN}}G=k$. Then, $b^{\DF{G}}_p(G)=\btwo{p}(G,(G_n);\FF)=0$ for $p\geq k-1$ and every residual chain $(G_n)$.
\end{prop}

\begin{proof}
Let $X$ be a finite model for a $K(G,1)$.  As explained in \cite[Theorem~3.2]{HaulmarkSchreve2022minimal} we may adapt the proof of \cite[Theorem~1.6]{Sauer2016} to apply to $k$-dimensional aspherical simplicial complexes.  In particular, for a residual chain $(G_n)_{n\in\N}$ we obtain a sequence of covers $X_n\to X$, such that the number of $p$-cells in $X_n$ grows sublinearly in $[G:G_n]$.  Since 
\[
    b^{\DF{G}}_p(X_n)= [G:G_n] \cdot b^{\DF{G}}_p(G) \leqslant \mc |I_n(X_k)|
\] 
where $I_p(X_k)$ is the set of $p$-cells of $X_n$.  But now, as $k$ tends to infinity, the left hand side of the equation grows linearly, and the right hand side of the equation grows sublinearly.  This is only possible if $b^{\DF{G}}_n(G)=0$.  The statement concerning $\btwo{p}(G,(G_n);\FF)$ is analogous.
\end{proof}

Let $X$ be a finite CW complex with a piecewise Riemannian metric $g$.  Fix a basepoint $x_0$ in the universal cover $\widetilde X$ and let $\widetilde{g}$ be the pull-back metric.  The \emph{volume entropy} of $(X,g)$ is
\[\mathrm{ent}(X,g)\coloneqq\lim_{t\to \infty}\frac{1}{t}\mathrm{Vol}(B_{x_0}(t),\widetilde g). \]
The \emph{minimal volume entropy} of $X$ is
\[\omega(X)\coloneqq \inf_g \mathrm{ent}(X,g)\mathrm{Vol}(X,g)^{1/\dim X} \]
where $g$ varies over all piecewise Riemannian metrics.  The invariant was originally defined for Riemannian manifolds in \cite{Gromov1982volBC}.

Suppose $G$ is a group admitting a finite $K(G,1)$.  The \emph{minimal volume entropy} of $G$ is
\[\omega(G)\coloneqq \inf(\omega(X)) \]
where $X$ ranges over all finite models of a $K(G,1)$ such that $\dim X=\gd(G)$.

There are few calculations of minimal volume entropy of groups which are not fundamental groups of aspherical manifolds in literature.  To date there is the work of Babenko--Sabourau \cite{BabenkoSabourau2021} on which computations for free-by-cyclic groups \cite{BregmanClay2021} and RAAGs \cite{HaulmarkSchreve2022minimal,Li2022amenable} have been completed.

We say $G$ has \emph{uniformly uniform exponential growth} if each subgroup either has uniform exponential growth bounded below by some constant $\omega_0>1$ or is virtually abelian.  Note that this property is sometimes called \emph{uniform uniform exponential growth} or \emph{locally uniform exponential growth}.

\begin{cor}\label{cor min vol}
Let $\FF$ be a skew field.  Let $G$ be a residually finite group  of type $\mathsf{F}$, and suppose $\DF{G}$ exists. If $G$ has uniformly uniform exponential growth and is not $\DF{G}$-acyclic, then $\omega(G)>0$.
\end{cor}
\begin{proof}
This follows from \cite[Paragraph after Theorem~3.3]{HaulmarkSchreve2022minimal} swapping out their use of \cite[Theorem~3.3]{HaulmarkSchreve2022minimal} for our \Cref{AMN cat vs DFG}.
\end{proof}

\begin{cor}\label{min vol BBL}
Let $\FF$ be a skew field and let $\varphi \colon A_L \rightarrow \Z$ be an epimorphism. Suppose $BB_L^\varphi$ is of type $\mathsf{F}$.  If 
$\bigoplus_{v \in L^\mathsf{a}} \widetilde{H}_{p-1}(\lk(v); \Z)\neq 0$,
    then $\omega(BBL^\varphi_L)>0$.
\end{cor}
\begin{proof}
This follows from \Cref{cor min vol} and the fact that a right-angled Artin group has strongly uniform exponential growth by \cite{Baudisch1981}.
\end{proof}

We conjecture that the converse of the last corollary holds.

\begin{conj}
Let $\FF$ be a skew field and let $\varphi \colon A_L \rightarrow \Z$ be an epimorphism. Suppose $BB_L^\varphi$ is of type $\mathsf{F}$.  If
    $\bigoplus_{v \in L^\mathsf{a}} \widetilde{H}_{p-1}(\lk(v); \Z)=0$, 
    then $\omega(BBL^\varphi_L)=0$.
\end{conj}

\bibliographystyle{alpha}
\bibliography{bib}

\end{document}